\documentclass[12pt]{amsart}

\usepackage[utf8]{inputenc}
\usepackage[english]{babel}

\usepackage[letterpaper,top=2cm,bottom=2cm,left=3cm,right=3cm,marginparwidth=1.75cm]{geometry}

\usepackage{amsmath}
\usepackage{graphicx}
\usepackage[colorlinks=true, allcolors=blue]{hyperref}

\usepackage{amssymb} 
\usepackage{comment}
\usepackage{indentfirst}
\usepackage{csquotes}
\usepackage{theoremref}
\usepackage{hyperref}
\usepackage{geometry}

\hypersetup{hidelinks, colorlinks=true, allcolors=cyan, pdfstartview=Fit, breaklinks=true}

\usepackage{amsthm}

\theoremstyle{plain}
\newtheorem{theorem}{Theorem}
\numberwithin{theorem}{section}
\newtheorem{corollary}[theorem]{Corollary}
\newtheorem*{corollary*}{Corollary}
\newtheorem*{Example*}{Example}

\newtheorem{lemma}[theorem]{Lemma}
\newtheorem{proposition}[theorem]{Proposition}

\newtheorem{conjecture}[theorem]{Conjecture}
\theoremstyle{definition}

\newtheorem*{def*}{Definition}
\newtheorem*{theorem*}{Theorem}

\newtheorem*{definition*}{Definition}

\theoremstyle{remark}
\newtheorem*{remark}{Remark}

\newcommand{\bracket}[1]{\left( #1 \right)}
\newcommand{\floor}[1]{\left\lfloor #1 \right\rfloor}
\newcommand{\modulo}[3]{#1\equiv#2\ \bracket{\mathrm{mod}\ #3}}
\newcommand{\notmodulo}[3]{#1\not\equiv#2\ \bracket{\mathrm{mod}\ #3}}

\newcommand{\qPochhammer}[3]{(#1;#2)_\infty^{#3}}
\newcommand{\Sh}[1]{\mathrm{Sh}_{#1}}

\numberwithin{equation}{section}

\title{\textbf{Arithmetic properties of overpartitions}}
\author{QI-YANG ZHENG}
\date{} 

\address{Department of Mathematics, Sun Yat-sen University(Zhuhai Campus), Zhuhai}
\email{zhengqy29@mail2.sysu.edu.cn}

\begin{document}

\begin{abstract}
    The primary focus of this paper is overpartitions, a type of partition that plays a significant role in $q$-series theory. In 2006, Treneer discovered an explicit infinite family of congruences of overpartitions modulo $5$. In our research, we have identified explicit infinite families of congruences of overpartitions modulo $3,7,11$. This work reveals the connection between overpartitions and half-integral modular forms.
\end{abstract}

\maketitle

~

\section{Introduction}

An overpartition of $n$ is an ordered sequence of non-increasing positive integers that sum to $n$, where the first occurrence of each integer may be overlined. For example, $3$ has eight overpartitions:
$$3,\bar3,2+1,\bar2+1,2+\bar1,\bar2+\bar1,1+1+1,\bar1+1+1.$$

\noindent
Obtaining the generating function of the overpartition function is straightforward.
\begin{equation}
    \abovedisplayskip=1em
    \notag
    \sum_{n=0}^\infty \bar p(n)q^n = \prod_{n=1}^\infty \frac{1+q^n}{1-q^n}.
    \belowdisplayskip=1em
\end{equation}

\noindent
The results in \cite{corteel2004overpartitions} show that several finite products appearing in $q$-series possess natural interpretations in terms of overpartitions. Furthermore, overpartitions have been found to play a central role in bijective proofs of Ramanujan’s $_1\psi_1$ summation and the $q$-Gauss summation.

Our paper shows that overpartitions yield results of another type. First, we introduce Ramanujan-type congruences. Let $p(n)$ denote the number of unrestricted partitions of $n$. Ramanujan discovered that,
\begin{equation}
    \abovedisplayskip=1em
    \notag
    \begin{aligned}
        {p(5n+4)}&\equiv{0}\ (\mathrm{mod}\ {5}), \\
        {p(7n+5)}&\equiv{0}\ (\mathrm{mod}\ {7}), \\
        {p(11n+6)}&\equiv{0}\ (\mathrm{mod}\ {11)}.
    \end{aligned}
    \belowdisplayskip=1em
\end{equation}

\noindent
Such congruences also appear in the context of overpartitions. In 2006, Treneer discovered an explicit infinite family of congruences modulo $5$ \cite[Prop. 1.4]{treneer2006congruences}.

\begin{theorem}[Treneer]
    Let $\modulo{Q}{4}{5}$ be prime. Then
    \begin{equation}
        \abovedisplayskip=2pt
        \notag
        \modulo{\bar p(5Q^3n)}{0}{5}
        \belowdisplayskip=2pt
    \end{equation}

    \noindent
    for all $n$ coprime to $Q$.
\end{theorem}

\noindent
It is natural to ask whether $5$ is the only such special prime. Surprisingly, a similar phenomenon occurs for moduli $3, 7$, and $11$. In the case of modulus $3$, we discover two infinite families of congruences.

\begin{theorem}
    \label{Mod 3 congruences}
    Let $\modulo{Q}{5}{6}$ be prime. Then
    \begin{enumerate}
        \item $\modulo{\bar p(3Q^3n)}{0}{3}$ for all $n$ coprime to $Q$.
        \item $\modulo{\bar p(Q^3n)}{0}{3}$ for all $n$ coprime to $Q$, provided $\modulo{n}{-1}{3}$.
    \end{enumerate}
\end{theorem}

\noindent
Note that the two types of congruences mentioned above are completely disjoint. For modulus $7$, we identify three infinite families of congruences.

\begin{theorem}
    \label{Mod 7 congruences}
    Let $\modulo{Q}{3,5,6}{7}$ be prime. Then
    \begin{equation}
        \abovedisplayskip=2pt
        \notag
        \modulo{\bar p(7Q^3n)}{0}{7}
        \belowdisplayskip=2pt
    \end{equation}

    \noindent
    for all $n$ coprime to $Q$.
\end{theorem}

\noindent
Moving on to modulus $11$, we find one infinite family of congruences.

\begin{theorem}
    \label{Mod 11 congruences}
    Let $\modulo{Q}{10}{11}$ be prime. Then
    \begin{equation}
        \abovedisplayskip=2pt
        \notag
        \modulo{\bar p(11Q^3n)}{0}{11}
        \belowdisplayskip=2pt
    \end{equation}

    \noindent
    for all $n$ coprime to $Q$.
\end{theorem}

\noindent
Congruences modulo primes $m\geq13$ do exist. However, there seems to be no explicit form for $Q$ that satisfies the condition
\begin{equation}
    \abovedisplayskip=2pt
    \notag
    \modulo{\bar p(mQ^3n)}{0}{m}
    \belowdisplayskip=2pt
\end{equation}

\noindent
for all $n$ coprime to $Q$. In fact, we can state the following theorem:

\begin{theorem}
    \label{modulo large primes}
    Let $m\geq13$ be prime. Then there are positive proportion of primes $Q$ such that
    \begin{equation}
        \abovedisplayskip=2pt
        \notag
        \modulo{\bar p(mQ^3n)}{0}{m}
        \belowdisplayskip=2pt
    \end{equation}

    \noindent
    for all $n$ coprime to $mQ$.
\end{theorem}

\noindent
In the proofs of these theorems, we discover a close relationship between overpartitions and half-integral weight modular forms. The key to proving these congruences lies in identifying the corresponding modular form. However, finding such a modular form is not a straightforward task.

\section{Preliminaries}

First, we introduce the $U$ and $V$ operators for formal series. If $j$ is a positive integer, these operators are defined as follows:
\begin{equation}
    \abovedisplayskip=1em
    \notag
    \left( \sum_{n=0}^\infty a(n)q^n \right)\ |\ U(j):=\sum_{n=0}^\infty a(jn)q^n,
    \belowdisplayskip=1em
\end{equation}
\begin{equation}
    \abovedisplayskip=1em
    \notag
    \left( \sum_{n=0}^\infty a(n)q^n \right)\ |\ V(j):=\sum_{n=0}^\infty a(n)q^{jn}.
    \belowdisplayskip=1em
\end{equation}

\noindent
The behavior of these operators is described in the following proposition \cite[Proposition 2.22]{ono2004web}.

\begin{proposition}
    Let $f(z)\in M_{k}(\Gamma_0(N),\chi)$ with integral weight.
    \begin{enumerate}
        \item If $j\mid N$, then $f(z)\ |\ U(j)\in M_{k}(\Gamma_0(N),\chi)$.
        \item $f(z)\ |\ V(j)\in M_{k}(\Gamma_0(jN),\chi)$.
    \end{enumerate}
\end{proposition}

\noindent
Now, we introduce the Hecke operator on modular forms of integral weight. Let $Q$ be a prime. The Hecke operator $T(Q)$ is defined as follows:
\begin{equation}
    \abovedisplayskip=1em
    \notag
    \left( \sum_{n=0}^\infty a(n)q^n \right)\ |\ T(Q):=\sum_{n=0}^\infty \bracket{a(Qn)+\chi(Q)Q^{k-1}a\bracket{\frac nQ}}q^n,
    \belowdisplayskip=1em
\end{equation}

\noindent
where $a(n/Q)=0$ if $Q\nmid n$.

Let $f(z)\in M_{k}(\Gamma_0(N),\chi)$. Then, $f(z)\ |\ T(Q)\in M_{k}(\Gamma_0(N),\chi)$. If $f(z)$ is a cusp form, then so are $f(z)\ |\ U(j)$, $f(z)\ |\ V(j)$, and $f(z)\ |\ T(Q)$.

There are $U$ and $V$ operators for modular forms of half-integral weight. The following proposition describes the behavior of these operators on half-integral weight modular forms \cite[Proposition 3.7]{ono2004web}.

\begin{proposition}
    \label{U,V operators for half-integral weight modular forms}
    Let $f(z)\in M_{\lambda+\frac12}(\Gamma_0(4N),\chi)$.
    \begin{enumerate}
        \item If $j\mid N$, then $f(z)\ |\ U(j)\in M_{\lambda+\frac12}(\Gamma_0(4N),\bracket{\frac{4j}{\bullet}}\chi)$.
        \item $f(z)\ |\ V(j)\in M_{\lambda+\frac12}(\Gamma_0(4jN),\bracket{\frac{4j}{\bullet}}\chi)$.
    \end{enumerate}
\end{proposition}

\noindent
For primes $Q$, the half-integral weight Hecke operator $T(Q^2)$ for $f(z)\in M_{\lambda+\frac12}(\Gamma_0(4N),\chi)$ is defined as follows:
\begin{equation}
    \abovedisplayskip=1em
    \notag
    \begin{aligned}
        &\ \ \ \ \left( \sum_{n=0}^\infty a(n)q^n \right)\ |\ T(Q^2) \\
        &:=\sum_{n=0}^\infty \bracket{a(Q^2n)+\bracket{\frac{(-1)^\lambda n}{Q}}\chi(Q)Q^{\lambda-1}a(Q)+\bracket{\frac{(-1)^\lambda}{Q^2}}\chi(Q^2)Q^{2\lambda-1}a\bracket{\frac{n}{Q^2}}}q^n,
    \end{aligned}
    \belowdisplayskip=1em
\end{equation}

\noindent
where $f(z)\ |\ T(Q^2)\in M_{\lambda+\frac12}(\Gamma_0(4N),\chi)$. Moreover, if $f(z)\in M_{\lambda+\frac12}(\Gamma_0(4N),\chi)$ is a cusp form, then so are $f(z)\ |\ U(j)$, $f(z)\ |\ V(j)$, and $f(z)\ |\ T(Q^2)$.

It's worth recalling that Dedekind's eta function is defined as:
\begin{equation}
    \abovedisplayskip=1em
    \notag
    \eta(z)=q^\frac{1}{24}\prod_{n=1}^\infty (1-q^n),
    \belowdisplayskip=1em
\end{equation}

\noindent
where $q=e^{2\pi iz}$. Thus
\begin{equation}
    \abovedisplayskip=1em
    \notag
    \sum_{n=0}^\infty\bar p(n)q^n=\frac{\eta(2z)}{\eta^2(z)}.
    \belowdisplayskip=1em
\end{equation}

\noindent
If $m$ is a prime, we denote by $M_{\frac{k}{2}}(\Gamma_0(N),\chi)_m$ (respectively, $S_{\frac{k}{2}}(\Gamma_0(N),\chi)_m$) the $\mathbb{F}_m$-vector space obtained by reducing the $q$-expansions of modular forms (resp. cusp forms) in $M_{\frac{k}{2}}(\Gamma_0(N),\chi)$ (resp. $S_{\frac{k}{2}}(\Gamma_0(N),\chi)$) with integer coefficients modulo $m$.

At times, for convenience, we will use the notation $a\equiv_m b$ instead of $a\equiv b\ (\mathrm{mod}\ m)$.

~

The construction of modular forms requires the utilization of the following theorem \cite[Theorem 3]{gordon1993multiplicative}:

\begin{theorem}[Gordon-Hughes]
\label{eta-quotient}
Let
$$f(z)=\prod_{\delta|N}\eta^{r_\delta}(\delta z)$$

\noindent
be a $\eta$-quotient provided

~

\noindent
$\mathrm{(\romannumeral1)}$ $$\sum_{\delta|N}\delta r_\delta\equiv0\ (\mathrm{mod}\ 24);$$
$\mathrm{(\romannumeral2)}$ $$\sum_{\delta|N}\frac{Nr_\delta}{\delta}\equiv0\ (\mathrm{mod}\ 24);$$
$\mathrm{(\romannumeral3)}$ $$k:=\frac{1}{2}\sum_{\delta|N}r_\delta\in\mathbb{Z},$$

\noindent
then

$$f\left(\frac{az+b}{cz+d}\right)=\chi(d)(cz+d)^kf(z),$$

\noindent
for each$\begin{pmatrix}
     a & b\\
     c & d
    \end{pmatrix}\in\Gamma_0(N)$ and $\chi$ is a Dirichlet character $(\mathrm{mod}\ N)$ defined by
$$\chi(n):=\left( \frac{(-1)^k\prod_{\delta|N}\delta^{r_\delta}}{n} \right),\ if\ n>0\ and\ (n,6)=1.$$
\end{theorem}

If $f(z)$ is holomorphic (resp. vanishes) at all cusps of $\Gamma_0(N)$, then $f(z)\in M_k(\Gamma_0(N)$ $,\chi)$ (resp. $S_k(\Gamma_0(N),\chi)$), as $\eta(z)$ never vanishes on $\mathcal{H}$. The following theorem (cf. \cite{martin1996multiplicative}) provides a useful criterion for computing the orders of an $\eta$-quotient at all cusps of $\Gamma_0(N)$.

\begin{theorem}[Martin]

\label{order of cusp}
Let $c$, $d$, and $N$ be positive integers with $d\mid N$ and $(c,d)=1$. If $f(z)$ is an $\eta$-quotient that satisfies the conditions of Theorem \ref{eta-quotient}, then the order of vanishing of $f(z)$ at the cusp $c/d$ is
$$\frac{N}{24}\sum_{\delta|N}\frac{r_\delta(d^2,\delta^2)}{\delta(d^2,N)}.$$

\end{theorem}

\section{Modulo 3}

First, we introduce the Ramanujan theta functions:
\begin{equation}
    \abovedisplayskip=1em
    \notag
    \phi(q)=\sum_{n=-\infty}^\infty q^{n^2},
    \belowdisplayskip=1em
\end{equation}
\begin{equation}
    \abovedisplayskip=1em
    \notag
    \psi(q)=\sum_{n=0}^\infty q^{\frac{n(n+1)}{2}}.
    \belowdisplayskip=1em
\end{equation}

\noindent
We will use the notation as presented in \cite{andrews2010arithmetic}, 
\begin{equation}
    \abovedisplayskip=1em
    \notag
    a(q)=\phi(-q^3),
    \belowdisplayskip=1em
\end{equation}
\begin{equation}
    \abovedisplayskip=1em
    \notag
    b(q)=\frac{\qPochhammer{q}{q}{}\qPochhammer{q^6}{q^6}{2}}{\qPochhammer{q^2}{q^2}{}\qPochhammer{q^3}{q^3}{}},
    \belowdisplayskip=1em
\end{equation}

\noindent
where $\qPochhammer{a}{q}{}=\prod_{n=0}^\infty(1-aq^n)$. The following lemma is useful for obtaining some surprising identities \cite[Lemma 2.6 and 2.7]{andrews2010arithmetic}.

\begin{lemma}[Andrews-Hirschhorn-Sellers]
    \label{Lemma AHS}
    \begin{equation}
        \abovedisplayskip=1em
        \notag
        \phi(-q)=a(q^3)-2qb(q^3),
        \belowdisplayskip=1em
    \end{equation}
    \begin{equation}
        \abovedisplayskip=1em
        \notag
        a(q)^3-8qb(q)^3=\frac{\phi(-q)^4}{\phi(-q^3)}.
        \belowdisplayskip=1em
    \end{equation}
\end{lemma}

\noindent
We will begin by proving the following theorem, which provides important dissections.

\begin{theorem}
    \label{dissections mod 3}
    \begin{equation}
        \abovedisplayskip=1em
        \notag
        \sum_{n=0}^\infty\bar p(3n)q^n=\frac{\qPochhammer{q^2}{q^2}{4}\qPochhammer{q^3}{q^3}{6}}{\qPochhammer{q}{q}{8}\qPochhammer{q^6}{q^6}{3}},
        \belowdisplayskip=1em
    \end{equation}
    \begin{equation}
        \abovedisplayskip=1em
        \notag
        \sum_{n=0}^\infty\bar p(3n+1)q^n=2\frac{\qPochhammer{q^2}{q^2}{3}\qPochhammer{q^3}{q^3}{3}}{\qPochhammer{q}{q}{7}},
        \belowdisplayskip=1em
    \end{equation}
    \begin{equation}
        \abovedisplayskip=1em
        \notag
        \sum_{n=0}^\infty\bar p(3n+2)q^n=4\frac{\qPochhammer{q^2}{q^2}{2}\qPochhammer{q^6}{q^6}{3}}{\qPochhammer{q}{q}{6}}.
        \belowdisplayskip=1em
    \end{equation}
\end{theorem}

\begin{corollary}
    $$\modulo{\bar p(3n+2)}{0}{4}.$$
\end{corollary}

\begin{proof}
    By \cite[Theorem 1.2.]{oliver2013eta} we have
    \begin{equation}
        \abovedisplayskip=1em
        \notag
        \sum_{n=0}^\infty\bar p(n)q^n=\frac{\eta(2z)}{\eta^2(z)}=\frac{1}{\phi(-q)}.
        \belowdisplayskip=1em
    \end{equation}

    \noindent
    Thus by Lemma \ref{Lemma AHS} we have
    \begin{equation}
        \abovedisplayskip=1em
        \notag
        \sum_{n=0}^\infty\bar p(n)q^{\frac n3}=\frac{1}{\phi(-q^{1/3})}=\frac{1}{a(q)-2q^{1/3}b(q)}=\frac{a(q)^2+2q^{1/3}a(q)b(q)+4q^{2/3}b(q)^2}{a(q)^3-8qb(q)^3}.
        \belowdisplayskip=1em
    \end{equation}

    \noindent
    By comparing the coefficients on both sides, we obtain
    \begin{equation}
        \abovedisplayskip=1em
        \notag
        \sum_{n=0}^\infty\bar p(3n)q^n=\frac{a(q)^2}{a(q)^3-8qb(q)^3}=\frac{\qPochhammer{q^2}{q^2}{4}\qPochhammer{q^3}{q^3}{6}}{\qPochhammer{q}{q}{8}\qPochhammer{q^6}{q^6}{3}},
        \belowdisplayskip=1em
    \end{equation}
    \begin{equation}
        \abovedisplayskip=1em
        \notag
        \sum_{n=0}^\infty\bar p(3n+1)q^n=\frac{2a(q)b(q)}{a(q)^3-8qb(q)^3}=2\frac{\qPochhammer{q^2}{q^2}{3}\qPochhammer{q^3}{q^3}{3}}{\qPochhammer{q}{q}{7}},
        \belowdisplayskip=1em
    \end{equation}
    \begin{equation}
        \abovedisplayskip=1em
        \notag
        \sum_{n=0}^\infty\bar p(3n+2)q^n=\frac{4b(q)^2}{a(q)^3-8qb(q)^3}=4\frac{\qPochhammer{q^2}{q^2}{2}\qPochhammer{q^6}{q^6}{3}}{\qPochhammer{q}{q}{6}}.
        \belowdisplayskip=1em
    \end{equation}
\end{proof}

\noindent
Now we are able to prove Theorem \ref{Mod 3 congruences} (1).

\begin{proof}[Proof of Theorem \ref{Mod 3 congruences} (1)]
    Note that
    \begin{equation}
        \abovedisplayskip=1em
        \notag
        \sum_{n=0}^\infty\bar p(3n)q^n=\frac{\qPochhammer{q^2}{q^2}{4}\qPochhammer{q^3}{q^3}{6}}{\qPochhammer{q}{q}{8}\qPochhammer{q^6}{q^6}{3}}\equiv_3\frac{\qPochhammer{q}{q}{10}}{\qPochhammer{q^2}{q^2}{5}}=\phi^5(-q).
        \belowdisplayskip=1em
    \end{equation}

    \noindent
    It is well known that $\phi(q)\in M_{\frac{1}{2}}(\Gamma_0(4))$ (for example, see \cite[Proposition 1.41]{ono2004web}). Since $\phi(-q) = 2\phi(q)\ |\ U(2)\ |\ V(2) - \phi(q)$, by Proposition \ref{U,V operators for half-integral weight modular forms}, we obtain
    \begin{equation}
        \abovedisplayskip=1em
        \notag
        \phi(-q)\in M_{\frac12}(\Gamma_0(16)).
        \belowdisplayskip=1em
    \end{equation}

    \noindent
    So
    \begin{equation}
        \abovedisplayskip=1em
        \notag
        \phi^5(q)=\sum_{n=0}^\infty r_5(n)q^n\in M_{\frac52}(\Gamma_0(4)),
        \belowdisplayskip=1em
    \end{equation}

    \noindent
    where $r_s(n)$ denotes the number of ways to express $n$ as the sum of $s$ squares. By \cite[Lemma 5.1.]{cooper2002sums} we know that for each odd prime $Q$,
    \begin{equation}
        \abovedisplayskip=1em
        \label{r_5 is eigenform}
        r_5(Q^2n)+Q\bracket{\frac nQ}r_5(n)+Q^3r_5\bracket{\frac{n}{Q^2}}=(Q^3+1)r_5(n),
        \belowdisplayskip=1em
    \end{equation}

    \noindent
    thus
    \begin{equation}
        \abovedisplayskip=1em
        \notag
        (-1)^{Q^2n}r_5(Q^2n)+(-1)^nQ\bracket{\frac nQ}r_5(n)+(-1)^{n/Q^2}Q^3r_5\bracket{\frac{n}{Q^2}}=(-1)^n(Q^3+1)r_5(n),
        \belowdisplayskip=1em
    \end{equation}

    \noindent
    where $(-1)^{n/Q^2}=0$ if $Q^2\nmid n$. The equation above is equivalent to
    \begin{equation}
        \abovedisplayskip=1em
        \notag
        \phi^5(-q)\ |\ T(Q^2)=(Q^3+1)\phi^5(-q).
        \belowdisplayskip=1em
    \end{equation}

    \noindent
    So if $\modulo{Q}{5}{6}$, we have $\modulo{\phi^5(-q)\ |\ T(Q^2)}{0}{3}$, hence
    \begin{equation}
        \abovedisplayskip=1em
        \notag
        \modulo{\sum_{n=0}^\infty\bar p(3n)q^n\ |\ T(Q^2)}{0}{3},
        \belowdisplayskip=1em
    \end{equation}

    \noindent
    i.e.
    \begin{equation}
        \abovedisplayskip=1em
        \notag
        \modulo{\bar p(3Q^2n)+Q\bracket{\frac nQ}\bar p(n)+Q^3\bar p\bracket{\frac{n}{Q^2}}}{0}{3}
        \belowdisplayskip=1em
    \end{equation}

    \noindent
    for all $n$. Let $n=Qn$ with $(n,Q)=1$ to eliminate the latter two parts, obtaining
    \begin{equation}
        \abovedisplayskip=1em
        \notag
        \modulo{\bar p(3Q^3n)}{0}{3}
        \belowdisplayskip=1em
    \end{equation}

    \noindent
    for all $n$ coprime to $Q$.
\end{proof}

\noindent
The proof for Theorem \ref{Mod 3 congruences} (2) is more complicated. First, we need a theorem by Sturm \cite[Theorem 1]{sturm1987congruence}, which offers a useful criterion for determining when modular forms with integer coefficients become congruent to zero modulo a prime through finite computation. We will use the refined version of Sturm's Theorem, as presented in \cite[Theorem 2.58]{ono2004web}, which also applies to half-integral weight modular forms.

\begin{theorem}[Sturm]
\label{Sturm's theorem}
Suppose $f(z)=\sum_{n=0}^\infty a(n)q^n\in M_{\frac k2}(\Gamma_0(N),\chi)_m$ such that
$$a(n)\equiv0\ (\mathrm{mod}\ m)$$

\noindent
for all $n\leq \frac{kN}{24}\prod_{p|N}\left( 1+\frac1p \right)$. Then $a(n)\equiv0\ (\mathrm{mod}\ m)$ for all $n\in\mathbb{Z}$.
\end{theorem}

Now, let us briefly introduce the famous Shimura lift \cite{shimura1973modular}. The result we use is from \cite{niwa1975modular} (also referred to in \cite[Theorem 3.14]{ono2004web}).

\begin{theorem}[Niwa-Shimura]
    Suppose that $g(z)=\sum_{n=1}^\infty b(n)q^n\in S_{\lambda+\frac12}(\Gamma_0(4N),\chi)$ with $\lambda\geq1$. Let $t$ be a positive square-free integer, and define the Dirichlet character $\psi_t$ by $\psi_t(n)=\chi(n)\bracket{\frac{-1}{n}}^\lambda\bracket{\frac tn}$. If complex numbers $A_t(n)$ are defined by
    \begin{equation}
        \abovedisplayskip=1em
        \notag
        \sum_{n=1}^\infty\frac{A_t(n)}{n^s}:=\sum_{n=1}^\infty\frac{\psi_t(n)}{n^{s-\lambda+1}}\cdot\sum_{n=1}^\infty\frac{b(tn^2)}{n^s},
        \belowdisplayskip=1em
    \end{equation}

    \noindent
    then
    \begin{equation}
        \abovedisplayskip=1em
        \notag
        \Sh{t}(g(z)):=\sum_{n=1}^\infty A_t(n)q^n
        \belowdisplayskip=1em
    \end{equation}
    
    \noindent
    is a modular form in $M_{2\lambda}(\Gamma_0(2N), \chi^2)$. If $\lambda\geq 2$, then $\Sh{t}(g(z))$ is a cusp form.
\end{theorem}

\noindent
With tools above, we are able to prove Theorem \ref{Mod 3 congruences} (2).

\begin{proof}[Proof of Theorem \ref{Mod 3 congruences} (2)]
    By Theorem \ref{dissections mod 3} we have
    \begin{equation}
        \abovedisplayskip=1em
        \notag
        \sum_{n=0}^\infty\bar p(3n+1)q^{3n+1}=2q\frac{\qPochhammer{q^6}{q^6}{3}\qPochhammer{q^9}{q^9}{3}}{\qPochhammer{q^3}{q^3}{7}}\equiv_32q\qPochhammer{q}{q}{6}\qPochhammer{q^2}{q^2}{9}=2\eta^6(z)\eta^9(2z).
        \belowdisplayskip=1em
    \end{equation}

    \noindent
    By a table listed in \cite{aydin2023hecke}, we know that for each odd prime $Q$, $\eta^6(z)\eta^9(2z)\in S_{\frac{15}{2}}(\Gamma_0(16))$ is a Hecke eigenform for operator $T(Q^2)$. Let
    \begin{equation}
        \abovedisplayskip=1em
        \notag
        \eta^6(z)\eta^9(2z)=\sum_{n=1}^\infty b(n)q^n,
        \belowdisplayskip=1em
    \end{equation}

    \noindent
    then for each odd prime $Q$,
    \begin{equation}
        \abovedisplayskip=1em
        \notag
        \sum_{n=1}^\infty \bracket{b(Q^2n)+\bracket{\frac{-n}{Q}}Q^6b(n)+Q^{13}b\bracket{\frac{n}{Q^2}}}q^n=\mu(Q)\sum_{n=1}^\infty b(n)q^n,
        \belowdisplayskip=1em
    \end{equation}

    \noindent
    for some constant $\mu(Q)$. By comparing the coefficient of the term $q^1$, we obtain
    \begin{equation}
        \abovedisplayskip=1em
        \notag
        b(Q^2)+\bracket{\frac{-1}{Q}}Q^6b(1)=\mu(Q)b(1).
        \belowdisplayskip=1em
    \end{equation}

    \noindent
    Since $b(1)=1$, we have
    \begin{equation}
        \abovedisplayskip=1em
        \notag
        \mu(Q)=b(Q^2)+\bracket{\frac{-1}{Q}}Q^6.
        \belowdisplayskip=1em
    \end{equation}

    \noindent
    Unfortunately, we have less information about $b(Q^2)$. Next, we will use the Shimura lift to study the properties of $b(Q^2)$. Let
    \begin{equation}
        \abovedisplayskip=1em
        \notag
        \Sh{1}(\eta^6(z)\eta^9(2z))=\sum_{n=1}^\infty A_1(n)q^n\in S_{14}(\Gamma_0(8)).
        \belowdisplayskip=1em
    \end{equation}

    \noindent
    By definition, we know that
    \begin{equation}
        \abovedisplayskip=1em
        \notag
        \begin{aligned}
            A_1(n)&=\sum_{d\mid n}\bracket{\frac{-1}{d}}\chi_0(d)d^6b\bracket{\frac{n^2}{d^2}} && (\mathrm{where}\ \chi_0\ \mathrm{is\ the\ trivial\ character\ mod}\ 16) \\
            &=\sum_{\genfrac{}{}{0pt}{}{d\mid n}{d\ \mathrm{odd}}}\bracket{\frac{-1}{d}}d^6b\bracket{\frac{n^2}{d^2}}.
        \end{aligned}
        \belowdisplayskip=1em
    \end{equation}

    \noindent
    Note that
    \begin{equation}
        \abovedisplayskip=1em
        \notag
        A_1(Q)=b(Q^2)+\bracket{\frac{-1}{Q}}Q^6.
        \belowdisplayskip=1em
    \end{equation}

    \noindent
    Since our purpose is to study $A_1(Q)$, we do not need information for $A_1(n)$ with $(n,6)>1$, so we are going to cancel these terms. Precisely, we will study
    \begin{equation}
        \abovedisplayskip=1em
        \notag
        \sum_{\genfrac{}{}{0pt}{}{n=1}{\modulo{n}{1,5}{6}}}^\infty A_1(n)q^n.
        \belowdisplayskip=1em
    \end{equation}

    \noindent
    Note that
    \begin{equation}
        \abovedisplayskip=1em
        \label{decomposition of A_1(n)}
        \begin{aligned}
            &\ \ \ \ \sum_{\genfrac{}{}{0pt}{}{n=1}{\modulo{n}{1,5}{6}}}^\infty A_1(n)q^n\\
            &=\sum_{n=1}^\infty A_1(n)q^n-\sum_{\genfrac{}{}{0pt}{}{n=1}{\modulo{n}{0}{2}}}^\infty A_1(n)q^n-\sum_{\genfrac{}{}{0pt}{}{n=1}{\modulo{n}{0}{3}}}^\infty A_1(n)q^n+\sum_{\genfrac{}{}{0pt}{}{n=1}{\modulo{n}{0}{6}}}^\infty A_1(n)q^n,
        \end{aligned}
        \belowdisplayskip=1em
    \end{equation}

    \noindent
    where
    \begin{equation}
        \abovedisplayskip=1em
        \notag
        \sum_{\genfrac{}{}{0pt}{}{n=1}{\modulo{n}{0}{6}}}^\infty A_1(n)q^n=\sum_{n=1}^\infty A_1(6n)q^{6n}=\sum_{n=1}^\infty A_1(n)q^n\ |\ U(6)\ |\ V(6).
        \belowdisplayskip=1em
    \end{equation}

    \noindent
    If we view $\sum_{n=1}^\infty A_1(n)q^n$ as an element in $S_{14}(\Gamma_0(24))$, then
    \begin{equation}
        \abovedisplayskip=1em
        \notag
        \sum_{n=1}^\infty A_1(6n)q^{6n}=\sum_{n=1}^\infty A_1(n)q^n\ |\ U(6)\in S_{14}(\Gamma_0(24)).
        \belowdisplayskip=1em
    \end{equation}

    \noindent
    Hence
    \begin{equation}
        \abovedisplayskip=1em
        \notag
        \sum_{\genfrac{}{}{0pt}{}{n=1}{\modulo{n}{0}{6}}}^\infty A_1(n)q^n=\sum_{n=1}^\infty A_1(n)q^n\ |\ U(6)\ |\ V(6)\in S_{14}(\Gamma_0(144)).
        \belowdisplayskip=1em
    \end{equation}

    \noindent
    Similar argument can be applied to other terms of \eqref{decomposition of A_1(n)}. Finally, we obtain
    \begin{equation}
        \abovedisplayskip=1em
        \notag
        \sum_{\genfrac{}{}{0pt}{}{n=1}{\modulo{n}{1,5}{6}}}^\infty A_1(n)q^n\in S_{14}(\Gamma_0(144)).
        \belowdisplayskip=1em
    \end{equation}

    \noindent
    Use Theorem \ref{eta-quotient} and \ref{order of cusp}, we find that $\eta^4(6z)\in S_2(\Gamma_0(36))$. Since Eisenstein series
    \begin{equation}
        \abovedisplayskip=1em
        \notag
        E_4(z)=1+240\sum_{n=1}^\infty\sigma_3(n)q^n\in M_4(\Gamma_0(1))
        \belowdisplayskip=1em
    \end{equation}

    \noindent
    and $\modulo{E_4(z)}{1}{3}$, we have $\eta^4(6z)E_4^3(z)\in S_{14}(\Gamma_0(144))$. Thus $\eta^4(6z)\in S_{14}(\Gamma_0(144))_3$.

    Now we can apply Theorem \ref{Sturm's theorem} to show that
    \begin{equation}
        \abovedisplayskip=1em
        \label{Mod 3 relation}
        \modulo{\sum_{\genfrac{}{}{0pt}{}{n=1}{\modulo{n}{1,5}{6}}}^\infty A_1(n)q^n}{\eta^4(6z)}{3},
        \belowdisplayskip=1em
    \end{equation}

    \noindent
    well beyond the Sturm bound of $336$. Note that $\eta^4(6z)$ have nonzero coefficients only at $q^{6n+1}$ terms, we obtain $\modulo{A_1(n)}{0}{3}$ for all $\modulo{n}{5}{6}$. Thus if odd prime $\modulo{Q}{5}{6}$, we have
    \begin{equation}
        \abovedisplayskip=1em
        \notag
        \mu(Q)=b(Q^2)+\bracket{\frac{-1}{Q}}Q^6=A_1(Q)\equiv_30.
        \belowdisplayskip=1em
    \end{equation}

    \noindent
    Thus for prime $\modulo{Q}{5}{6}$,
    \begin{equation}
        \abovedisplayskip=1em
        \notag
        \modulo{\sum_{\genfrac{}{}{0pt}{}{n=0}{\modulo{n}{1}{3}}}^\infty\bar p(n)q^n\ |\ T(Q^2)}03.
        \belowdisplayskip=1em
    \end{equation}

    \noindent
    Or equivalently,
    \begin{equation}
        \abovedisplayskip=1em
        \notag
        \modulo{\bar p(Q^2n)+\bracket{\frac{-n}{Q}}Q^6\bar p(n)+Q^{13}\bar p\bracket{\frac{n}{Q^2}}}03
        \belowdisplayskip=1em
    \end{equation}

    \noindent
    satisfies for all $\modulo{n}{1}{3}$. Let $n=Qn$ with $(n,Q)=1$ and $\modulo{n}{-1}{3}$ to eliminate the latter two parts, obtaining
    \begin{equation}
        \abovedisplayskip=1em
        \notag
        \modulo{\bar p(Q^3n)}{0}{3}
        \belowdisplayskip=1em
    \end{equation}

    \noindent
    for all $(n,Q)=1$ and $\modulo{n}{-1}{3}$.
\end{proof}

\begin{remark}
    We can compute that $\dim S_{14}(\Gamma_0(8))=11$ (see, for example, \cite[Section 3.9]{diamond2005first}). Using Theorem \ref{eta-quotient} and Theorem \ref{order of cusp}, we can find a set of basis
    \begin{equation}
        \abovedisplayskip=1em
        \notag
        \begin{aligned}
            &\{\beta_1,\cdots,\beta_{11}\}=\\
            &\left\{
            \frac{\eta^{32}(z)}{\eta^4(2z)},\frac{\eta^{44}(4z)}{\eta^{16}(8z)},\eta^{20}(2z)\eta^8(4z),
            \eta^8(2z)\eta^{20}(4z),\frac{\eta^{32}(4z)}{\eta^4(2z)},\eta^{20}(4z)\eta^8(8z), \right.\\
            &\left. \eta^4(2z)\eta^8(4z)\eta^{16}(8z),\frac{\eta^{20}(4z)\eta^{16}(8z)}{\eta^8(2z)},
            \frac{\eta^8(4z)\eta^{24}(8z)}{\eta^4(2z)},\frac{\eta^{32}(8z)}{\eta^4(4z)},
            \frac{\eta^4(2z)\eta^{40}(8z)}{\eta^{16}(4z)}
            \right\}.
        \end{aligned}
        \belowdisplayskip=1em
    \end{equation}

    \noindent
    A straightforward computation reveals that
    \begin{equation}
        \abovedisplayskip=1em
        \notag
        \begin{aligned}
            \sum_{n=1}^\infty A_1(n)q^n&=\beta_1+96\beta_2-2304\beta_3-65536\beta_5-24576\beta_6+393216\beta_8-6291456\beta_{10}\\
           &\equiv_3 \beta_1-\beta_5.
        \end{aligned}
        \belowdisplayskip=1em
    \end{equation}

    \noindent
    Using the above congruence to prove \eqref{Mod 3 relation} is simpler than using the definition of the Shimura lift.
\end{remark}

\section{Modulo 7}

The case modulo $7$ is quite different from modulo $3$, since the explicit expression for $\sum\bar p(7n)q^n$ is very complicated and hard to analyse. First we need to show that $\sum\bar p(7n)q^n$ is a modular form, then we can use Sturm's Theorem to find a modular form which equals to $\sum\bar p(7n)q^n$ in the sense modulo $7$, so that we can avoid analyse the explicit expression for $\sum\bar p(7n)q^n$.

\begin{theorem}
    \label{sum bar p(mn) q^n is a modular form}
    For each prime $m\geq3$, let $m'=(m\ \mathrm{mod}\ 8)$, then
    \begin{equation}
        \abovedisplayskip=1em
        \notag
        \sum_{n=0}^\infty\bar p(mn)q^n\in M_{(8+m')\frac{m-1}{2}-\frac12}(\Gamma_0(16))_m.
        \belowdisplayskip=1em
    \end{equation}
\end{theorem}

\begin{proof}
    We define an $\eta$-quotient by
    \begin{equation}
        \abovedisplayskip=1em
        \notag
        f(m;z)=\frac{\eta(2z)}{\eta^2(z)}\eta^a(mz)\eta^b(2mz)\equiv_m\eta^{am-2}(z)\eta^{bm+1}(2z),
        \belowdisplayskip=1em
    \end{equation}

    \noindent
    where $a=2m'-8$ and $b=16-m'$. One can use Theorem \ref{eta-quotient} and \ref{order of cusp} to show that
    \begin{equation}
        \abovedisplayskip=1em
        \notag
        \eta^{am-2}(z)\eta^{bm+1}(2z)\in S_{4m+\frac{mm'-1}{2}}(\Gamma_0(2)).
        \belowdisplayskip=1em
    \end{equation}

    \noindent
    In fact, it has order $m$ at the cusp $\infty$ and order $(mm'-1)/8$ at the cusp $0$. On the other hand,
    \begin{equation}
        \abovedisplayskip=1em
        \notag
        \frac{\eta(2z)}{\eta^2(z)}\eta^a(mz)\eta^b(2mz)=\sum_{n=0}^\infty\bar p(n)q^{n+m}\cdot\qPochhammer{q^m}{q^m}{a}\qPochhammer{q^{2m}}{q^{2m}}{b}.
        \belowdisplayskip=1em
    \end{equation}

    \noindent
    Since $\modulo{U(m)}{T(m)}{m}$, we have
    \begin{equation}
        \abovedisplayskip=1em
        \label{after U(m) || T(m)}
        \frac{\eta(2z)}{\eta^2(z)}\eta^a(mz)\eta^b(2mz)\ |\ U(m)\equiv_m\eta^{am-2}(z)\eta^{bm+1}(2z)\ |\ T(m).
        \belowdisplayskip=1em
    \end{equation}

    \noindent
    As for the left hand side of \eqref{after U(m) || T(m)}, we have
    \begin{equation}
        \abovedisplayskip=1em
        \notag
        \frac{\eta(2z)}{\eta^2(z)}\eta^a(mz)\eta^b(2mz)\ |\ U(m)=\sum_{\genfrac{}{}{0pt}{}{n=0}{\modulo{n}{0}{m}}}^\infty \bar p(n)q^{\frac{n+m}{m}}\cdot\qPochhammer{q}{q}{a}\qPochhammer{q^2}{q^2}{b}.
        \belowdisplayskip=1em
    \end{equation}

    \noindent
    One can use Theorem \ref{eta-quotient} and \ref{order of cusp} to show that $\eta^8(z)\eta^8(2z)\in S_8(\Gamma_0(2))$ has order $1$ at all cusps. So as for the right hand side of \eqref{after U(m) || T(m)}, we have
    \begin{equation}
        \abovedisplayskip=1em
        \notag
        \eta^{am-2}(z)\eta^{bm+1}(2z)\ |\ T(m)=\eta^8(z)\eta^8(2z)g(m;z),
        \belowdisplayskip=1em
    \end{equation}

    \noindent
    where $g(m;z)\in M_{4m-8+\frac{mm'-1}{2}}(\Gamma_0(2))$. Combining the above two equation to show that
    \begin{equation}
        \abovedisplayskip=1em
        \notag
        \sum_{\genfrac{}{}{0pt}{}{n=0}{\modulo{n}{0}{m}}}^\infty \bar p(n)q^{\frac{n}{m}}\equiv_m\eta^{8-a}(z)\eta^{8-b}(2z)g(m;z)=\frac{\eta^{2(8-m')}(z)}{\eta^{8-m'}(2z)}g(m;z),
        \belowdisplayskip=1em
    \end{equation}

    \noindent
    where
    \begin{equation}
        \abovedisplayskip=1em
        \notag
        \frac{\eta^{2(8-m')}(z)}{\eta^{8-m'}(2z)}=\phi^{8-m'}(-q)\in M_{\frac{8-m'}{2}}(\Gamma_0(16)).
        \belowdisplayskip=1em
    \end{equation}

    \noindent
    Hence
    \begin{equation}
        \abovedisplayskip=1em
        \notag
        \sum_{n=0}^\infty\bar p(mn)q^n\in M_{(8+m')\frac{m-1}{2}-\frac12}(\Gamma_0(16))_m.
        \belowdisplayskip=1em
    \end{equation}
\end{proof}

\noindent
We need more tools to proceed. Let
\begin{equation}
    \abovedisplayskip=1em
    \notag
    \phi_{s,t}(q)=\phi(q)^s(2q^{\frac14}\psi(q^2))^t.
    \belowdisplayskip=1em
\end{equation}

\noindent
If $4\mid t$, we may write
\begin{equation}
    \abovedisplayskip=1em
    \notag
    \phi_{s,t}(q)=\sum_{n=0}^\infty\phi_{s,t}(n)q^n.
    \belowdisplayskip=1em
\end{equation}

\noindent
Let $r_{s,t}(n)$ denote the number of representations of $n$ as a sum of $s$ even squares and $t$ odd squares, considering both negative numbers and order. Then
\begin{equation}
    \abovedisplayskip=1em
    \label{phi as sum of squares}
    \sum_{n=0}^\infty r_{s,t}(n)q^n=\binom{s+t}{s}\phi_{s,t}(q^4).
    \belowdisplayskip=1em
\end{equation}

\noindent
We will require the following fact about modular forms of half-integral weight.

\begin{theorem}[{\cite[page 184, Prop. 4.]{koblitz2012introduction}}]
    \label{basis of half-integral weight and level 4}
    $M_{\frac{2k+1}{2}}(\Gamma_0(4))$ is the vector space consisting of all linear combinations of $\phi_{2k+1-4j,4j}(q), j=0,1,\cdots,\floor{k/2}$.
\end{theorem}

\noindent
Now we are going to prove an identity of divisor functions.
\begin{theorem}
    \label{divisor functions identity}
    \begin{equation}
        \abovedisplayskip=1em
        \notag
        \sum_{i=1}^n \sigma(2i-1)\sigma(2n-2i+1)=\sum_{\genfrac{}{}{0pt}{}{d\mid n}{n/d\ \mathrm{odd}}}d^3.
        \belowdisplayskip=1em
    \end{equation}
\end{theorem}

\begin{proof}
    By \cite[page 6]{ono1995representation} we have
    \begin{equation}
        \abovedisplayskip=1em
        \notag
        q\psi^4(q^2)=\sum_{n=0}^\infty\sigma(2n+1)q^{2n+1}
        \belowdisplayskip=1em
    \end{equation}

    \noindent
    and \cite[page 8]{ono1995representation}
    \begin{equation}
        \abovedisplayskip=1em
        \notag
        q^2\psi^8(q^2)=\sum_{n=1}^\infty\sum_{\genfrac{}{}{0pt}{}{d\mid n}{n/d\ \mathrm{odd}}}d^3q^{2n}.
        \belowdisplayskip=1em
    \end{equation}

    \noindent
    Comparing the coefficients of both sides gives the desired result.
\end{proof}

\begin{proof}[Proof of Theorem \ref{Mod 7 congruences}]
    By Theorem \ref{sum bar p(mn) q^n is a modular form} we have
    \begin{equation}
        \abovedisplayskip=1em
        \notag
        \sum_{n=0}^\infty \bar p(7n)q^n\in M_{\frac{89}{2}}(\Gamma_0(16))_7.
        \belowdisplayskip=1em
    \end{equation}

    \noindent
    Thus
    \begin{equation}
        \abovedisplayskip=1em
        \notag
        \sum_{n=0}^\infty (-1)^n\bar p(7n)q^n=2\sum_{n=0}^\infty \bar p(7n)q^n\ |\ U(2)\ |\ V(2)-\sum_{n=0}^\infty \bar p(7n)q^n\in M_{\frac{89}{2}}(\Gamma_0(32))_7.
        \belowdisplayskip=1em
    \end{equation}

    \noindent
    On the other hand, since Eisenstein series
    \begin{equation}
        \abovedisplayskip=1em
        \notag
        E_6(z)=1-504\sum_{n=1}^\infty \sigma_5(n)q^n\in M_6(\Gamma_0(1))
        \belowdisplayskip=1em
    \end{equation}

    \noindent
    and $\modulo{E_6(z)}{1}{7}$, we have
    \begin{equation}
        \abovedisplayskip=1em
        \notag
        E_6^7(z)(\phi^5(q)-2\phi_{1,4}(q))\in M_{\frac{89}{2}}(\Gamma_0(32)).
        \belowdisplayskip=1em
    \end{equation}

    \noindent
    One can use Theorem \ref{Sturm's theorem} to show that
    \begin{equation}
        \abovedisplayskip=1em
        \notag
        \modulo{\sum_{n=0}^\infty (-1)^n\bar p(7n)q^n}{\phi^5(q)-2\phi_{1,4}(q)}7,
        \belowdisplayskip=1em
    \end{equation}

    \noindent
    well beyond the Sturm bound of $178$. So we can view 
    \begin{equation}
        \abovedisplayskip=1em
        \notag
        \sum_{n=0}^\infty (-1)^n\bar p(7n)q^n\in M_{\frac52}(\Gamma_0(4))_7.
        \belowdisplayskip=1em
    \end{equation}

    \noindent
    Now we are going to analyse $(\phi^5(q)-2\phi_{1,4}(q))\ |\ T(Q^2)$ for odd prime $Q$. Since $M_{\frac52}(\Gamma_0(4))$ is of dimension $2$, we can write
    \begin{equation}
        \abovedisplayskip=1em
        \label{equation for T(Q^2)}
        (\phi^5(q)-2\phi_{1,4}(q))\ |\ T(Q^2) = c_0\phi^5(q) + c_1\phi_{1,4}(q)
        \belowdisplayskip=1em
    \end{equation}

    \noindent
    for some constants $c_0,c_1$. Since The above equation is equivalent to
    \begin{equation}
        \abovedisplayskip=1em
        \label{phi_1,4 | T(Q^2)}
        2\phi_{1,4}(q)\ |\ T(Q^2) = (Q^3+1-c_0)\phi^5(q) - c_1\phi_{1,4}(q)
        \belowdisplayskip=1em
    \end{equation}

    \noindent
    Comparing the constant term of both sides of \eqref{phi_1,4 | T(Q^2)}, obtaining $c_0=Q^3+1$. To obtain $c_1$, we may rewrite \eqref{phi_1,4 | T(Q^2)} via the definition of Hecke operator to obtain
    \begin{equation}
        \abovedisplayskip=1em
        \notag
        2\phi_{1,4}(Q^2n)+2Q\bracket{\frac{n}{Q}}\phi_{1,4}(n)+2Q^3\phi_{1,4}\bracket{\frac{n}{Q^2}}=-c_1\phi_{1,4}(n).
        \belowdisplayskip=1em
    \end{equation}

    \noindent
    Thus
    \begin{equation}
        \abovedisplayskip=1em
        \notag
        2\phi_{1,4}(Q^2)+2Q\phi_{1,4}(1)=-c_1\phi_{1,4}(1).
        \belowdisplayskip=1em
    \end{equation}

    \noindent
    Since $\phi_{1,4}(1)=16$, we have $\phi_{1,4}(Q^2)+16Q=-8c_1$. So next we need to compute $\phi_{1,4}(Q^2)$. By \eqref{phi as sum of squares} we have
    \begin{equation}
        \abovedisplayskip=1em
        \notag
        \sum_{n=0}^\infty r_{1,4}(n)q^n=5\sum_{n=0}^\infty \phi_{1,4}(n)q^{4n}.
        \belowdisplayskip=1em
    \end{equation}

    \noindent
    Thus $5\phi_{1,4}(Q^2)=r_{1,4}(4Q^2)$. Since $\modulo{4Q^2}{4}{8}$, we notice that if
    \begin{equation}
        \abovedisplayskip=1em
        \notag
        4Q^2=x_1^2+x_2^2+x_3^2+x_4^2+x_5^2,
        \belowdisplayskip=1em
    \end{equation}

    \noindent
    then some one of $x_i$ must be a multiple of $4$ and others are odd. So
    \begin{equation}
        \abovedisplayskip=1em
        \notag
        \begin{aligned}
            r_{1,4}(4Q^2)&=5r_{0,4}(4Q^2)+10\sum_{\genfrac{}{}{0pt}{}{i>1}{4\mid i}}r_{0,4}(4Q^2-i^2)\\
            &=5r_{0,4}(4Q^2)+10\sum_{i=1}^{\floor{Q/2}}r_{0,4}(4Q^2-16i^2).
        \end{aligned}
        \belowdisplayskip=1em
    \end{equation}

    \noindent
    Note that the number of representations of $8n+4$ as a sum of $4$ positive odd squares equals the number of representations of $n$ as a sum of $4$ triangle numbers. Moreover, the number of representations of $n$ as a sum of $4$ triangle numbers is equal to $\sigma(2n+1)$ \cite[(1.15)]{hirschhorn2005number}. Thus $r_{0,4}(8n+4)=16\sigma(2n+1)$, the coefficient $16$ occurs since we count square of negative odd numbers. Consequently,
    \begin{equation}
        \abovedisplayskip=1em
        \notag
        \begin{aligned}
            \phi_{1,4}(Q^2)&=16\sigma(Q^2)+32\sum_{i=1}^{(Q-1)/2}\sigma(Q^2-4i^2)\\
            &= 16\sigma(Q^2)+32\sum_{i=1}^{(Q-1)/2}\sigma((2i-1)(2Q-2i+1))\\
            &=16\sigma(Q^2)+16\sum_{i=1}^{(Q-1)/2}\sigma((2i-1)(2Q-2i+1))\\
            &\ \ \ \ \ \ \ \ \ \ \ \ \ \ \ +16\sum_{i=(Q+3)/2}^{Q}\sigma((2i-1)(2Q-2i+1))\\
            &=16\sum_{i=1}^{Q}\sigma((2i-1)(2Q-2i+1)).
        \end{aligned}
        \belowdisplayskip=1em
    \end{equation}

    \noindent
    Since $(2i-1,2Q-2i+1)=(2i-1,2Q)=(2i-1,Q)=1$ when $i\neq(Q+1)/2$, we have
    \begin{equation}
        \abovedisplayskip=1em
        \notag
        \phi_{1,4}(Q^2)=16\sigma(Q^2)-16\sigma^2(Q)+16\sum_{i=1}^{Q}\sigma(2i-1)\sigma(2Q-2i+1).
        \belowdisplayskip=1em
    \end{equation}

    \noindent
    By Theorem \ref{divisor functions identity} we have
    \begin{equation}
        \abovedisplayskip=1em
        \notag
        \begin{aligned}
            \phi_{1,4}(Q^2)&=16(1+Q+Q^2)-16(1+Q)^2+16(1+Q^3)\\
            &=16(Q^3-Q+1).
        \end{aligned}
        \belowdisplayskip=1em
    \end{equation}

    \noindent
    Recalling that we have computed $\phi_{1,4}(Q^2)+16Q=-8c_1$, so $c_1=-2(Q^3+1)$. Now \eqref{equation for T(Q^2)} becomes
    \begin{equation}
        \abovedisplayskip=1em
        \label{phi_1,4 is eigenform}
        (\phi^5(q)-2\phi_{1,4}(q))\ |\ T(Q^2) = (Q^3+1)(\phi^5(q)-2\phi_{1,4}(q)).
        \belowdisplayskip=1em
    \end{equation}

    \noindent
    Note that $\modulo{Q^3+1}{0}{7}$ for $\modulo{Q}{3,5,6}{7}$, so for these primes $Q$ we have
    \begin{equation}
        \abovedisplayskip=1em
        \notag
        \modulo{\sum_{n=0}^\infty (-1)^n\bar p(7n)q^n\ |\ T(Q^2)}{0}{7},
        \belowdisplayskip=1em
    \end{equation}

    \noindent
    i.e.
    \begin{equation}
        \abovedisplayskip=1em
        \notag
        \modulo{(-1)^{Q^2n}\bar p(7Q^2n)+(-1)^nQ\bracket{\frac{n}{Q}}\bar p(7n)+(-1)^{n/Q^2}Q^{3}\bar p\bracket{\frac{7n}{Q^2}}}07
        \belowdisplayskip=1em
    \end{equation}

    \noindent
    satisfies for all $n$. Let $n=Qn$ with $(n,Q)=1$ to eliminate the latter two parts, obtaining
    \begin{equation}
        \abovedisplayskip=1em
        \notag
        \modulo{\bar p(7Q^3n)}{0}{7}
        \belowdisplayskip=1em
    \end{equation}

    \noindent
    for all $(n,Q)=1$.
\end{proof}

\begin{corollary}
    For each odd prime $Q$, $\phi_{1,4}(q)$ is an eigenform for Hecke operators $T(Q^2)$, and the eigenvalue is $Q^3+1$.
\end{corollary}

\begin{proof}
    It immediately follows from \eqref{phi_1,4 is eigenform}.
\end{proof}

\section{Modulo 11}

By Theorem \ref{basis of half-integral weight and level 4} we know that a set of basis of $M_{\frac92}(\Gamma_0(4))$ is $\phi^9(q)$, $\phi_{5,4}(q)$, and $\phi_{1,8}(q)$. The following theorem describes how Hecke operators act on these basis \cite[Theorem 6.5.]{cooper2002sums}.

\begin{theorem}[Cooper]
    \label{Hecke operator act on M_9/2}
    Let $Q$ be an odd prime, $\alpha=Q^7+1$, $\beta=\theta(p)$, then
    \begin{equation}
        \abovedisplayskip=1em
        \notag
        \begin{pmatrix}
            \phi^9(q)\ |\ T(Q^2) \\
            \phi_{5,4}(q)\ |\ T(Q^2) \\
            \phi_{1,8}(q)\ |\ T(Q^2)
        \end{pmatrix}
        =\frac{1}{17}
        \begin{pmatrix}
            17\alpha & -2\alpha+2\beta & 2\alpha-2\beta \\
            0 & 16\alpha+\beta & \alpha-\beta \\
            0 & 16\alpha-16\beta & \alpha+16\beta \\
        \end{pmatrix}
        \begin{pmatrix}
            \phi^9(q) \\
            \phi_{5,4}(q) \\
            \phi_{1,8}(q)
        \end{pmatrix},
        \belowdisplayskip=1em
    \end{equation}

    \noindent
    one can see definition of $\theta(p)$ in \cite[(2.1)]{cooper2002sums}. We will not use the value of $\theta(p)$ in this paper.
\end{theorem}

\begin{proof}[Proof of Theorem \ref{Mod 11 congruences}]
    By Theorem \ref{sum bar p(mn) q^n is a modular form} we have
    \begin{equation}
        \abovedisplayskip=1em
        \notag
        \bar p(11Q^3n)\in M_{\frac{109}{2}}(\Gamma_0(16))_{11}.
        \belowdisplayskip=1em
    \end{equation}

    \noindent
    Thus
    \begin{equation}
        \abovedisplayskip=1em
        \notag
        \sum_{n=0}^\infty (-1)^n\bar p(11n)q^n=2\sum_{n=0}^\infty \bar p(11n)q^n\ |\ U(2)\ |\ V(2)-\sum_{n=0}^\infty \bar p(11n)q^n\in M_{\frac{109}{2}}(\Gamma_0(32))_{11}.
        \belowdisplayskip=1em
    \end{equation}

    \noindent
    On the other hand, since Eisenstein series
    \begin{equation}
        \abovedisplayskip=1em
        \notag
        E_{10}(z)=1-264\sum_{n=1}^\infty \sigma_9(n)q^n\in M_{10}(\Gamma_0(1))
        \belowdisplayskip=1em
    \end{equation}

    \noindent
    and $\modulo{E_{10}(z)}{1}{11}$, we have
    \begin{equation}
        \abovedisplayskip=1em
        \notag
        E_{10}^5(z)(\phi^9(q)-2\phi_{5,4}(q))\in M_{\frac{109}{2}}(\Gamma_0(32)).
        \belowdisplayskip=1em
    \end{equation}

    \noindent
    One can use Theorem \ref{Sturm's theorem} to show that
    \begin{equation}
        \abovedisplayskip=1em
        \notag
        \modulo{\sum_{n=0}^\infty (-1)^n\bar p(11n)q^n}{\phi^9(q)-2\phi_{5,4}(q)}{11},
        \belowdisplayskip=1em
    \end{equation}

    \noindent
    well beyond the Sturm bound of $218$. So we can view 
    \begin{equation}
        \abovedisplayskip=1em
        \notag
        \sum_{n=0}^\infty (-1)^n\bar p(11n)q^n\in M_{\frac92}(\Gamma_0(4))_{11}.
        \belowdisplayskip=1em
    \end{equation}

    \noindent
    For odd prime $Q$, by Theorem \ref{Hecke operator act on M_9/2} we obtain
    \begin{equation}
        \abovedisplayskip=1em
        \notag
        \begin{aligned}
            (\phi^9(q)-2\phi_{5,4}(q))\ |\ T(Q^2)&=\alpha\phi^9(q)-2\alpha\phi_{5,4}(q)\\
            &=(Q^7+1)(\phi^9(q)-2\phi_{5,4}(q)).
        \end{aligned}
        \belowdisplayskip=1em
    \end{equation}

    \noindent
    Since $\modulo{Q^7+1}{0}{11}$ for $\modulo{Q}{10}{11}$, so for these primes $Q$ we have
    \begin{equation}
        \abovedisplayskip=1em
        \notag
        \modulo{(\phi^9(q)-2\phi_{5,4}(q))\ |\ T(Q^2)}{0}{11},
        \belowdisplayskip=1em
    \end{equation}

    \noindent
    so for these primes $Q$,
    \begin{equation}
        \abovedisplayskip=1em
        \notag
        \modulo{\sum_{n=0}^\infty (-1)^n\bar p(11n)q^n\ |\ T(Q^2)}{0}{11},
        \belowdisplayskip=1em
    \end{equation}

    \noindent
    i.e.
    \begin{equation}
        \abovedisplayskip=1em
        \notag
        \modulo{(-1)^{Q^2n}\bar p(11Q^2n)+(-1)^nQ^3\bracket{\frac{n}{Q}}\bar p(11n)+(-1)^{n/Q^2}Q^{7}\bar p\bracket{\frac{11n}{Q^2}}}0{11}
        \belowdisplayskip=1em
    \end{equation}

    \noindent
    satisfies for all $n$. Let $n=Qn$ with $(n,Q)=1$ to eliminate the latter two parts, obtaining
    \begin{equation}
        \abovedisplayskip=1em
        \notag
        \modulo{\bar p(11Q^3n)}{0}{11}
        \belowdisplayskip=1em
    \end{equation}

    \noindent
    for all $(n,Q)=1$.
\end{proof}

\section{Modulo other primes}

For odd primes $m\leq11$, we have already prove that for each odd prime $\modulo{Q}{-1}{m}$,
\begin{equation}
    \abovedisplayskip=1em
    \label{pattern for odd primes leq 11}
    \modulo{\bar p(mQ^3n)}{0}{m}
    \belowdisplayskip=1em
\end{equation}

\noindent
for all $(n,Q)=1$. However, this fail for $m=13$, since $\modulo{\bar p(13\cdot103^3\cdot3)}{12}{13}$. Moreover, \eqref{pattern for odd primes leq 11} seems to fail for primes $\geq13$.

However, for primes $\geq13$, we still can prove that there are ininitely many primes $Q$ such that
\begin{equation}
    \abovedisplayskip=2pt
    \notag
    \modulo{\bar p(mQ^3n)}{0}{m}
    \belowdisplayskip=2pt
\end{equation}

\noindent
for all $n$ coprime to $mQ$. First we need a theorem due to Serre.

\begin{theorem}[Serre]
\label{Serre's theorem}
The set of primes $Q\equiv-1\ (\mathrm{mod}\ Nm)$ such that
$$f\ |\ T(Q)\equiv0\ (\mathrm{mod}\ m)$$

\noindent
for each $f(z)\in M_k(\Gamma_0(N),\psi)_m$ has positive density, where $T(Q)$ denotes the usual Hecke operator acting on $M_k(\Gamma_0(N),\psi)$.
\end{theorem}

\noindent
Some application of Shimura lift genealize this result (see \cite[3.30]{ono2004web}).

\begin{theorem}
\label{Serre's theorem refined}
The set of primes $Q\equiv-1\ (\mathrm{mod}\ 4Nm)$ such that
$$f\ |\ T(Q^2)\equiv0\ (\mathrm{mod}\ m)$$

\noindent
for each $f(z)\in S_{\lambda+\frac12}(\Gamma_0(4N),\psi)_m$ has positive density, where $T(Q^2)$ denotes the usual Hecke operator acting on $S_{\lambda+\frac12}(\Gamma_0(4N),\psi)$.
\end{theorem}

\noindent
Although the theorem states that such primes have a positive proportion in the arithmetic sequence $4Nmn-1$, when we actually find some examples of congruences, we shall test all primes.

By Theorem \ref{sum bar p(mn) q^n is a modular form}, $\sum \bar p(mn)q^n$ may be a non-cusp form. Unfortunately, we cannot directly use Theorem \ref{Serre's theorem refined} to prove the existence of congruences modulo primes $\geq 13$. However, we can still test whether $\sum \bar p(mn)q^n ,|, T(Q^2) \equiv_m 0$, but there is less chance of finding such an example.

The next theorem shows that we can transform a non-cusp form into a cusp form without changing many coefficients.

\begin{theorem}
    \label{cuspation}
    Suppose that $m\geq 3$ and $f=\sum_{n=0}^\infty a(n)q^n\in M_{\lambda+\frac12}(\Gamma_0(4N),\psi)_m$, then

    \begin{enumerate}
        \item If $m\geq5$,
        \begin{equation}
            \abovedisplayskip=1em
            \notag
            \sum_{\genfrac{}{}{0pt}{}{n=0}{\notmodulo{n}{0}{m}}}^\infty a(n)q^n\in S_{\lambda+\frac{m^2}2}\bracket{\Gamma_0\bracket{\frac{4Nm^2}{(N,m)}},\psi}_m.
            \belowdisplayskip=1em
        \end{equation}

        \item If $m=3$,
        \begin{equation}
            \abovedisplayskip=1em
            \notag
            \sum_{\genfrac{}{}{0pt}{}{n=0}{\notmodulo{n}{0}{3}}}^\infty a(n)q^n\in S_{\lambda+\frac{25}2}\bracket{\Gamma_0\bracket{\frac{36N}{(3,N)}},\psi}_3.
            \belowdisplayskip=1em
        \end{equation}
    \end{enumerate}
\end{theorem}

\begin{proof}
    According to proof of \cite[Prop. 3.5.]{treneer2006congruences}, we know that if there are integer $t$ such that
    \begin{equation}
        \abovedisplayskip=1em
        \notag
        \sum_{n=0}^\infty a(n)q^n\ |\ U(m^t)
        \belowdisplayskip=1em
    \end{equation}

    \noindent
    is holomorphic at all cusps $\frac{a}{cm^2}$, then
    \begin{equation}
        \abovedisplayskip=1em
        \notag
        \sum_{n=0}^\infty a(n)q^n\ |\ U(m^t)-\sum_{n=0}^\infty a(n)q^n\ |\ U(m^{t+1})\ |\ V(m)
        \belowdisplayskip=1em
    \end{equation}

    \noindent
    vanishes at all cusps $\frac{a}{cm^2}$. Since $f$ itself is holomorphic at all cusps, we can simply take $t=0$. Now
    \begin{equation}
        \abovedisplayskip=1em
        \label{vanishes at a/cm^2}
        \begin{aligned}
            \sum_{\genfrac{}{}{0pt}{}{n=0}{\notmodulo{n}{0}{m}}}^\infty a(n)q^n
            &=\sum_{n=0}^\infty a(n)q^n-\sum_{n=0}^\infty a(n)q^n\ |\ U(m)\ |\ V(m)\\
            &\in M_{\lambda+\frac{1}2}\bracket{\Gamma_0\bracket{\frac{4Nm^2}{(N,m)}},\psi}_m
        \end{aligned}
        \belowdisplayskip=1em
    \end{equation}

    \noindent
    vanishes at all cusps $\frac{a}{cm^2}$. By \cite[page 13]{treneer2006congruences} we know
    \begin{equation}
        \abovedisplayskip=1em
        \label{vanishes at a/c with m^2 nmid c}
        \begin{cases}
             1\equiv_m\frac{\eta^{m^2}(z)}{\eta(m^2z)}\in M_{\frac{m^2-1}{2}}(\Gamma_0(m^2)) & \text{ if } m\geq5, \\
             1\equiv_3\frac{\eta^{27}(z)}{\eta^3(9z)}\in M_{12}(\Gamma_0(9)) & \text{ if } m=3.
        \end{cases}
        \belowdisplayskip=1em
    \end{equation}

    \noindent
    vanishes at all cusps $\frac{a}{c}$ of $\Gamma_0(Nm^2)$ with $m^2\nmid c$. Multiplying the two modular form in \eqref{vanishes at a/cm^2} and \eqref{vanishes at a/c with m^2 nmid c} gives the desired result.
\end{proof}

\begin{proof}[Proof of Theorem \ref{modulo large primes}]
    It is a corollary of Theorem \ref{sum bar p(mn) q^n is a modular form}, \ref{Serre's theorem refined}, and \ref{cuspation}. 
\end{proof}

\noindent
While Ramanujan-type congruences modulo all primes $m$ do exist, it is important to note that discovering them may require extensive computations. We encourage interested readers to explore and seek examples of congruences modulo primes $\geq 13$.

It is natural to ask whether only primes $3,5,7,11$ are special. In fact, we may conjecture that

\begin{conjecture}
    Let $m$ be an odd prime. If for all odd primes $\modulo{Q}{-1}{m}$, we have
    \begin{equation}
        \abovedisplayskip=2pt
        \notag
        \modulo{\bar p(mQ^3n)}{0}{m}
        \belowdisplayskip=2pt
    \end{equation}

    \noindent
    for all $n$ coprime to $Q$, then $m=3,5,7,11$.
\end{conjecture}

\section*{Acknowledgement}

We would like to express our gratitude to OEIS for providing many valuable references.

\end{document}